\documentclass[12pt,reqno,final,a4paper]{article}

\usepackage{hyperref}
\usepackage{amsmath}
\usepackage{amsthm}
\usepackage{amssymb}
\usepackage{dsfont}
\usepackage{stmaryrd}
\usepackage{commath}
\usepackage{mathrsfs}
\usepackage{enumitem}
\usepackage{graphicx}
\usepackage{tikz}
\font\msbm=msbm10

\theoremstyle{plain}
\newtheorem{theorem}{Theorem}[section]
\newtheorem{lemma}[theorem]{Lemma}
\newtheorem{corollary}[theorem]{Corollary}
\newtheorem{proposition}[theorem]{Proposition}

\theoremstyle{definition}

\newtheorem{remark}[theorem]{Remark}
\def\mathbb#1{\hbox{\msbm{#1}}}

\newcommand{\field}[1]{\ensuremath{\mathds{#1}}}

\newcommand{\re}{\field R}\newcommand{\N}{\field N}
\newcommand{\zz}{\ensuremath{\mathbb{Z}}}\newcommand{\C}{\ensuremath{\mathbb{C}}}
\newcommand{\tor}{\field T} 

\newcommand{\Z}{\field Z}

\newcommand{\R}{\field R}
\newcommand{\T}{\field T}

\newcommand{\id}{\ensuremath{\mathrm{id}}}
\newcommand{\Id}{\ensuremath{\mathrm{Id}}}

\newcommand{\rank}{\ensuremath{\mathrm{rank}}}

\newcommand{\sign}{{ \ensuremath{\mbox{\rm sign} \,} }}

\newcommand{\be}{\begin{equation}}
\newcommand{\ee}{\end{equation}}
\newcommand{\beq}{\begin{eqnarray}}
\newcommand{\beqq}{\begin{eqnarray*}}
\newcommand{\eeq}{\end{eqnarray}}
\newcommand{\eeqq}{\end{eqnarray*}}

\DeclareMathOperator{\vol}{vol}

\newcommand{\Sphere}{\ensuremath{\mathds S^{d-1}}}

\hypersetup{
   pdfauthor={Thomas Kühn, Sebastian Mayer, Tino Ullrich},
   pdfkeywords={approximation numbers, entropy numbers},
   pdftitle={Counting via entropy: new preasymptotics for the approximation numbers of Sobolev embeddings}
}

\begin{document}

\title{Counting via entropy: new preasymptotics for the approximation numbers of Sobolev embeddings}

\author{Thomas K\"uhn, Sebastian Mayer, Tino Ullrich}
\maketitle

\begin{abstract}
We study the optimal linear $L_2$-approximation by operators of finite rank (i.e., approximation numbers) for the isotropic periodic Sobolev space $H^s(\T^d)$ of fractional smoothness on the $d$-torus. For a family of weighted norms, which penalize Fourier coefficients $\hat f(k)$ by a weight $w_{s,p}(k) = (1+\|k\|_p^p)^{s/p}$, $0 < p \leq \infty$, we prove that the $n$-th approximation number of the embedding $\Id: H^s(\T^d) \to L_2(\T^d)$ is characterized by the $n$-th (non-dyadic) entropy number of the embedding $\id: \ell_p^d \to \ell_{\infty}^d$ raised to the power $s$. From the known behavior of these entropy numbers we gain a complete understanding of the approximation numbers in terms of $n$ and $d$ for all $n \in \N$ and $d \in \N$.
\end{abstract}

\section{Introduction}\label{sec:intro}

Approximation numbers, also known as linear $n$-width, are one of the fundamental concepts in approximation theory. In
Hilbert space settings, they describe the worst-case error that occurs when we approximate a class of functions by
projecting them onto the optimal finite-dimensional subspace. Hence, approximation numbers are also of interest in the
numerical analysis of partial differential equations (PDE) as they provide reliable a-priori error estimates for
certain Galerkin methods. In this context, approximation numbers related to
isotropic Sobolev spaces, Sobolev spaces of mixed regularity, and spaces of Gevrey type appear naturally.
Subject of this paper are preasymptotic bounds for these approximation numbers, which substantially improve the known error bounds in high-dimensional settings.


Bounds, which describe the decay in the rank $n \in \N$ of the optimal
projection operator, have been known for decades for the aforementioned approximation numbers. In high-dimensional settings, where the functions to be
approximated depend on a large number of variables $d \in \N$, these classical bounds become problematic. They only inadequately
capture the effects of the approximation problem's dimensionality $d$. We state this issue precisely in the
subsequent Sections \ref{sec:intro:isotropic}, \ref{sec:intro:gevrey}, and \ref{sec:intro:enegry}. For the moment, we
only stress that this has consequences in at least two respects when the dimension $d$ becomes large. For one, the
classical bounds are trivial until $n$ is exponentially large, say $n > 2^d$. This severely limits their applicability.
Moreover, from the viewpoint of information-based complexity, it is impossible to determine the tractability of the
approximation problem rigorously. To address any of the two issues, a first step is to 
uncover how the equivalence constants in the classical error bounds depend on the dimension $d$. As it turns out this is
often not enough. In fact, one has to determine explicitly how the approximation numbers behave
\emph{preasymptotically}, that is, for small $n < 2^d$. This typically involves to find good estimates for complicated
combinatorial problems that evolve from the structure of the smoothness spaces. As we will see, the preasymptotic
behavior can be completely different from the asymptotic behavior.

\subsection{An abstract characterization result: counting via entropy}

The essence of this paper is that for certain relevant approximation numbers of periodic isotropic Sobolev spaces and
periodic spaces of Gevrey type, it is not necessary to compute preasymptotics by hand. Instead we exploit that the
approximation numbers can be determined by covering certain $\ell_p^d$-unit balls with $\ell_\infty^d$-balls, a problem
which is already well understood. In fact, this turns out to be just a special case of a general characterization
result. This holds true for Sobolev type spaces $H^{\mathbf{w}}(\T^d)$ defined on the $d$-torus $\T^d=[0,2\pi]^d$, where
the smoothness weights $\mathbf{w} = (w(k))_{k \in \Z^d}$ take a special form. The spaces $H^{\mathbf{w}}(\T^d)$ are
given by
\begin{align*} 
 H^{\mathbf{w}}(\T^d) = \big\{f \in L_2(\T^d): \sum_{k \in \Z^d} w(k)^2 |c_k(f)|^2 < \infty \big\}
\end{align*}
where $c_k(f)$ denotes the $k$th Fourier coefficient. The weights in the sequence $\mathbf{w}$ are of the form
\begin{align}\label{eq:general_weight}
w(k) = \varphi(\|k\|), \quad k \in \Z^d,
\end{align}
where $\|\cdot\|$ is a (quasi-)norm on $\R^d$ and $\varphi$ a univariate, monotonically increasing function $\varphi$
with $\varphi(0)=1$. The characterization results now states that the approximation numbers associated to the embedding
$\Id: H^{\mathbf{w}}(\T^d) \to L_2(\T^d)$ are bounded for all $n \in \N$ from above and below as 
\begin{align}\label{eq:intro_entropy} 
 1/\varphi(2/\varepsilon_n) \leq a_n(\Id: H^{\mathbf{w}}(\T^d) \to L_2(\T^d)) \leq
1/\varphi(1/(4\varepsilon_n)). 
\end{align}
Here, $\varepsilon_n = \varepsilon_n(\id: \ell_{\|\cdot\|}^d \to \ell_\infty^d)$ are \emph{entropy numbers}. They give
the smallest $\varepsilon > 0$ such that the $\ell_{\|\cdot\|}^d$-unit ball can be covered by $n$ $\ell_\infty^d$-balls
of radius $\varepsilon$. For the precise statement, see Theorem \ref{res:approx_entropy_mono}.


\subsection{Preasymptotics for isotropic Sobolev spaces}\label{sec:intro:isotropic}
The first concrete application of the abstract result \eqref{eq:intro_entropy} yields results for the isotropic
Sobolev space $H^s(\T^d)$ with fractional smoothness $s > 0$. 
Isotropic Sobolev regularity is the natural notion of regularity for solutions of general elliptic PDEs, typically the
solution  will be contained in $H^1(\T^d)$ or $H^2(\T^d)$. The space $H^s(\T^d)$ can be defined as $H^s(\T^d) = H^{\mathbf{w}_{s,2}}(\T^d)$, where the weight sequence $\mathbf{w}_{s,2}$ is given by $w_{s,2}(k) = (1+\|k\|_2^2)^{s/2}$. Note that the norm $\|f|H^s(\T^d)\| = \sqrt{\sum_{k \in \Z^d} w_{s,2}(k)^2 |c_k(f)|^2}$ is natural in the sense that if $s \in \N$, then this norm  is equivalent up to a constant in $s$ to the classical norm, which is defined in terms of the $L_2$-norms of the derivatives up to order $s$. 

Concerning the approximation numbers $a_n (\Id :
H^s(\T^d) \to L_2(\T^d))$,
the exact asymptotic decay in $n$ has been known for decades. In 1967, J. W. Jerome \cite{J67} proved that
\begin{align}\label{eq:appr_num_asymp}
 c_{s,d} \; n^{-s/d} \leq a_n (\Id : H^s(\T^d) \to L_2(\T^d)) \leq C_{s,d} \; n^{-s/d},
\end{align}
with constants $c_{s,d}$ and $C_{s,d}$ that were merely known to depend on the fractional smoothness $s$ and the
dimension $d$. For further references and historical remarks in this direction, see the monographs by
Temlyakov \cite{Te93} and Tikhomirov \cite{Ti90}.

In order to obtain preasymptotics for \eqref{eq:appr_num_asymp} and clarify the $d$-dependence of the constants, we not
only consider the weights $\mathbf{w}_{s,2}$ but the family of weights $\mathbf{w}_{s,p}$ given by
\begin{align}
\label{def:iso_weights}
\begin{split}
\begin{array}{rcl}
 w_{s,p}(k) =& (1 + \sum_{j=1}^d |k_j|^p)^{s/p} &\text{ if } 0 < p < \infty, \text{ and }\\
 w_{s,p}(k) =& \max(1,|k_1|,\dots,|k_d|)^s &\text{ if } p=\infty.
\end{array}
\end{split}
\end{align}
For $0<p<1$, the weights $\mathbf{w}_{s,p}$ can be interpreted as imposing a \emph{compressibility constraint} on the
Fourier frequency vectors; the less $k \in \Z^d$ is aligned with one of the coordinate axes, the stronger the penalty
through a large weight $w_{s,p}(k)$. The function spaces $H^{s,p}(\T^d) := H^{\mathbf{w}_{s,p}}(\T^d)$ coincide as sets
with the classical isotropic Sobolev space. Applying the abstract result \eqref{eq:intro_entropy}, we immediately obtain
Theorem \ref{res:approx_numbers}.

\begin{theorem}\label{res:approx_numbers}
For $0< p \leq \infty$ and $s > 0$, we have
\begin{align*}
a_n (\Id:H^{s,p}(\T^d) \to L_2(\T^d)) \asymp_{s,p}
\begin{cases}
1 &: 1 \leq n \leq d\\
\left[ \frac{\log(1+d/\log n)}{\log n}\right]^{s/p} &: d \leq n \leq 2^d\\
d^{-s/p} n^{-s/d} &: n \geq 2^{d}.
\end{cases}
\end{align*}
\end{theorem}
\noindent It is clearly visible how a smaller compressibility parameter $p$ makes the approximation problem easier by amplifying the preasymptotic logarithmic decay in $n$. The equivalence constants in Theorem \ref{res:approx_numbers} depend only on $s$ and $p$ and can be completely
controlled. In particular, we have the limit result
\[ 
 \lim_{n \to \infty} n^{s/d} a_n(\Id: H^{s,p}(\T^d) \to L_2(\T^d)) = (\vol(B_p^d))^{s/d} \asymp d^{-s/p}, 
\]
see Corollary \ref{res:iso_asymp_constant}. A consequence of Theorem \ref{res:approx_numbers} is that we face the
\emph{curse of dimensionality} in the strict sense of information-based complexity if and only if $p=\infty$. Otherwise,
the approximation problem is \emph{weakly tractable}, despite the slow asymptotic decay $n^{-s/d}$. For details, see
Section \ref{sec:tractability}.

\subsection{Spaces of Gevrey type and a connection to hyperbolic cross spaces}\label{sec:intro:gevrey}
The classes of smooth functions that are nowadays called \emph{Gevrey classes} were already introduced in 1918 by 
M. Gevrey \cite{G1918}, they occurred in a natural way in his research on partial differential equations.
Since then they have played an important role in numerous applications, in particular in connection with Cauchy
problems. 
The recent paper \cite{KP15} introduces the periodic \emph{spaces of Gevrey type} 
$G^{\alpha,\beta,p}(\T^d)$, $0<\alpha,\beta,p<\infty$, which
consist of all $f\in C^\infty(\mathbb{T}^d)$ such that the norm
$$
\Vert f |G^{\alpha,\beta,p}(\T^d)\Vert:=\left( \sum_{k\in\Z^d}\exp(2\beta\,\Vert k
\Vert_p^\alpha)|c_k(f)|^2\right)^{1/2}
$$
is finite. Here $c_k(f)$ denotes the Fourier coefficient with respect to the frequency vector $k=(k_1,...,k_d)\in \Z^d$,
defined in Section \ref{sec:preliminaries} below. For $0<\alpha<1$, the spaces $G^{\alpha,\beta,p}(\T^d)$ coincide with
the classical Gevrey classes and contain
non-analytic functions, while for $\alpha\ge 1$ all functions in $G^{\alpha,\beta,p}(\T^d)$ are analytic. Some more
background on Gevrey classes and references can be found in Section \ref{sec:gevrey}.

In Theorem \ref{res:gevrey} we prove lower and upper bounds for the approximation numbers $a_n$ of the embedding $\Id:
G^{\alpha,\beta,p}(\T^d) \to L_2(\T^d)$ for all $n \in \N$ and arbitrary parameter values $\alpha,\beta > 0$, $0<p\leq \infty$. Due to our proof technique we can determine the rate of convergence only up to a constant. However, for $0<p\leq \infty$ and $\alpha<\min\{1,p\}$, we at least obtain an indication for the correct asymptotic
behavior by the limit statement
\[ 
 \lim_{n \to \infty} a_n \cdot \exp(\lambda \beta n^{\alpha/d}) = 1,
\]
where $\lambda:=\vol(B_p^d)^{-\alpha/d}$, see Theorem \ref{limit:gevrey}. 

What concerns preasymptotics the bounds turn out to be rather surprising in the particular situation $\alpha = p$. For $1 \leq n \leq 2^d$, we
obtain the two-sided estimate
\begin{align}\label{GinL2}
 n^{-\frac{c_1(p) \beta}{\log(1+d/\log(n))}}
  \leq a_n(\Id: G^{p,\beta,p}(\T^d) \to L_2(\T^d)) \leq
  n^{-\frac{c_2(p) \beta}{\log(1+d/\log(n))}} \leq n^{-\frac{c_2(p) \beta}{\log(1+d)}}.
\end{align}
This estimate is almost identical to the preasymptotic estimate which has been obtained in the recent paper
\cite{KSU15} (see also \eqref{eq:mixed_quasi_polynomial_decaya} below) for approximation numbers of the embeddings
\[
 \Id: H^r_{\text{mix}}(\T^d) \to L_2(\T^d),
\]
where $H^r_{\text{mix}}(\T^d)$ is the Sobolev space of dominating mixed smoothness equipped with the norm
\[ 
 \|f | H^r_{\text{mix}}(\T^d)\| := \bigg[ \sum_{k \in \Z^d} |c_k(f)|^2 \prod_{j=1}^d (1+ |k_j|^2)^r \bigg]^{1/2}.
\]
It is rather counterintuitive that the approximation numbers behave almost identically in the preasymptotic range. After
all, the spaces of Gevrey type $G^{p,\beta,p}(\T^d)$ contain substantially smoother functions than the space
$H^r_{\text{mix}}(\T^d)$. We discuss this in more detail in Section \ref{sec:preasymptotic_similarity} and give at least
partial explanations for this odd phenomenon.

\subsection{Preasymptotics for embeddings into $\mathbf{H^s}$}\label{sec:intro:enegry} 

Yserentant \cite{Y2010} proved that eigenfunctions of the positive spectrum of the electronic Schr\"odinger
operator possess a dominating mixed regularity. To solve the electronic Schr\"odinger equation numerically, Galerkin methods combined with sparse grid techniques \cite{GH2007,GKZ2007} are widely used. The discussion in
Subsection \ref{sec:intro:galerkin} below shows that one is particularly interested in measuring the error in the {\it
energy space} $H^1$. From results in \cite{GK2009} it follows that 
\begin{align}\label{eq:Galerkin_energy} 
  c_d n^{-(r-1)} \leq a_n(H^r_{\text{mix}}(\T^d) \to  H^1(\T^d)) \leq C_d n^{-(r-1)},
\end{align}
with constants $c_d$, $C_d$ depending on the dimension $d$. In \cite{DU13,GK2009} it has been observed that
$C_d = d^20.97515^d$. This result suggests that the truncation problem even gets easier with a
growing number of electrons. However, as \cite{DU13} shows, the constant $C_d$
can only be chosen as above for exponentially large $n > (1+\gamma)^d$. This raises the question how the approximation numbers in \eqref{eq:Galerkin_energy} behave preasymptotically.

Unfortunately, the abstract result \eqref{eq:intro_entropy} cannot be applied to obtain preasymptotics for
\eqref{eq:Galerkin_energy}, since the space $H^r_{\text{mix}}(\T^d)$ cannot be written as a space $H^{\mathbf{w}}(\T^d)$
with a weight sequence $\mathbf{w}$ of the form \eqref{eq:general_weight}. However, the observations described in the previous subsection and results in \cite{KSU15} give an indication for the preasymptotic behavior. The results in \cite[Thm. 4.9, 4.10, 4.17]{KSU15} provide the two-sided estimate
\begin{align}\label{eq:mixed_quasi_polynomial_decaya} 
 2^{-r} \left(\frac{1}{2n}\right)^{\frac{r}{2 + \log(1/2+d/\log(n))}} \leq a_n ( \Id: H^{r}_{\text{mix}}(\T^d)
\to L_2(\T^d) ) \leq \left( \frac{e^2}{n} \right)^{\frac{r}{4 + 2\log_2 d}}
\end{align}
in the preasymptotic range $1 \leq n \leq 4^d$.
With the coincidence $$a_n(\Id:H^{r}_{\text{mix}}(\T^d) \to H^s_{\text{mix}}(\T^d)) = a_n(\Id:H^{r-s}_{\text{mix}}(\T^d)
\to L_2(\T^d)),$$
provided $r>s>0$, we obtain from \eqref{eq:mixed_quasi_polynomial_decaya} by embedding 
\begin{align}\label{eq:mixed_quasi_polynomial_decay2} 
a_n ( \Id: H^{r}_{\text{mix}}(\T^d)
\to H^s(\T^d) ) \leq \left( \frac{e^2}{n} \right)^{\frac{r-s}{4 + 2\log_2 d}}.
\end{align}
The connection between spaces of Gevrey type and spaces of dominating mixed smoothness sketched in Subsection \ref{sec:intro:gevrey}
(see also Subsection \ref{sec:preasymptotic_similarity} below), might be useful to refine the result
\eqref{eq:mixed_quasi_polynomial_decay2}. Indeed, for spaces of Gevrey type we obtain the following result as a
consequence of our abstract technique. For $1 \leq n \leq 2^d$, we have
\begin{equation}
  c_1\Lambda(n,d)^{s/2} n^{-\frac{c_1 \beta}{\gamma(n,d)}}
 \leq a_n(\Id:G^{2,r/2,2}(\T^d) \to H^s(\T^d)) \leq
  c_2 \Lambda(n,d)^{s/2} n^{-\frac{c_2 r}{\gamma(n,d)}},
\end{equation}
where $\Lambda(n,d) = \frac{\log(n)}{\gamma(n,d)}$ and $\gamma(n,d) = \log(1+d/\log(n))$; compare with \eqref{GinL2}.

\subsection{Approximation numbers and Galerkin methods}\label{sec:intro:galerkin}
To conclude this introduction, let us outline the connection between approximation numbers and reliable a-priori error
estimates for Galerkin methods. Consider a general elliptic variational problem in $H^s = H^s(\T^d)$, which is given by
a bilinear symmetric form $a: H^s \times H^s \to \R$ and a right-hand side $f \in H^{-s}$. The bilinear symmetric form
is assumed to satisfy, for any $u,v \in H^s$,
\[
a(u,v) \le \mu_1 \|\,u\, | H^s \| \|\,v\, | H^s\| \  {\rm and} \ a(u,u) \ge \mu_2 \|\,u\, | H^s \|^2.
\]
Under this assumption, $a(\cdot,\cdot)$ generates the so called \emph{energy norm equivalent} to the norm of $H^s$.
The problem now is to find an element $u \in H^s$ such that
\begin{equation} \label{eq[a(u,v)]}
a(u,v)
\ = \
(f,v) \ {\rm for \ all} \ v \in H^s.
\end{equation} 
In order to get an approximate numerical solution Galerkin methods solve the same problem on a finite dimensional
subspace $V_h$ in $H^s$,
\begin{equation} \label{eq[a(u_h,v)]}
a(u_h,v)
\ = \
(f,v) \ {\rm for \ all} \ v \in V_h.
\end{equation} 
By the Lax-Milgram theorem \cite{LM54}, the problems \eqref{eq[a(u,v)]} and \eqref{eq[a(u_h,v)]} have unique solutions 
$u^*$ and $u^*_h$, respectively, which by C\'ea's lemma \cite{C64}, satisfy the inequality 
\begin{align}\label{eq:cea}
\|u^* - u^*_h\,|H^s\|
\ \le (\mu_1/\mu_2) \inf_{v \in V_h} \|u^* - v\,|H^s\|.
\end{align}
The naturally arising question is how to choose the optimal $n$-dimensional subspace
$V_h$ and linear finite element approximation algorithms such that the right-hand side in \eqref{eq:cea} becomes as
smalls as possible. Under the assumption that the solution $u^*$ is contained in the unit ball of some smoothness space
$U \subset H^s$, the minimal right-hand side  in \eqref{eq:cea} is bounded from above by the approximation number
$a_n(\Id: U \to H^s)$. Summarizing,
\begin{align*}
\|u^* - u^*_h \,|H^s\|
\ \le (\mu_1/\mu_2) a_n(\Id: U \to H^s)
\end{align*}
gives a worst-case a-priori error estimate for the optimal $n$-dimensional subspace $V_h$.

\section{Preliminaries}\label{sec:preliminaries}

\paragraph*{Notation} As usual, the set $\N$ denotes the natural numbers, $\zz$ the integers and
$\re$ the real numbers. By $\tor$ we denote the torus represented by the interval $[0,2\pi]$ where opposite points are
identified. A function $f:\T^d\to \C$ is $2\pi$-periodic in every component. If $f\in L_2(\T^d)$ then the Fourier
coefficient $c_k(f)$ with respect to the frequency vector $k=(k_1,...,k_d) \in \Z^d$ is given by
$$
  c_k(f) = (2\pi)^{-d/2}\int_{\T^d} f(x)e^{-ik\cdot x}\,dx\,.
$$  
For a real number $a$ we put $a_+ := \max\{a,0\}$.
The symbol $d$ is always reserved for the dimension in $\Z^d$, $\R^d$, $\N^d$, and $\T^d$.
For $0<p\leq \infty$ and $x\in \R$ we denote $\|x\|_p = (\sum_{i=1}^d |x_i|^p)^{1/p}$ with the
usual modification in the case $p=\infty$. We write $\ell_p^d$ for $\R^d$ equipped with the norm $\|\cdot\|_p$. By
$B_p^d$ we denote the closed unit ball of $\ell_p^d$. When we write $\log$, we always mean the logarithm to base
$2$.
If $X$ and $Y$ are two Banach spaces, the norm
of an element $x$ in $X$ will be denoted by $\|x|X\|$ and the norm of an operator
$A:X \to Y$ is denoted by $\|A:X\to Y\|$. The symbol $X \hookrightarrow Y$ indicates that the
embedding operator is continuous.

\paragraph*{Approximation numbers} Let $X,Y$ be two (quasi-)Banach spaces.
The $n$-th approximation number of an operator $T: X \to Y$ is defined by
\be\label{def:approximation_number}
 \begin{split}
a_n (T : X \to Y) &:= \inf\limits_{\rank A<n} \,\sup\limits_{\|f|X\|\leq 1}\|\, Tf -Af|Y\|\\
&= \inf\limits_{\rank A<n}\|T-A : X \to Y\|.
 \end{split}
\ee

\paragraph*{Covering and entropy numbers}
Let $A \subset \R^d$. An \emph{$\varepsilon$-net} for $A$ is a discrete set of points $x_1,\dots,x_n$ in $\R^d$ such that
\[
 A \subseteq \bigcup_{i=1}^n (x_i + \varepsilon B_{\infty}^d).
\]
The \emph{covering number} $N_{\varepsilon}(A)$ is the minimal natural number $n$ such that there is an $\varepsilon$-net for $A$.
Inverse to the covering numbers $N_{\varepsilon}(A)$ are the (non-dyadic) \emph{entropy numbers}
\[
 \varepsilon_n(A, \ell_\infty^d) := \inf \{ \varepsilon > 0: N_{\varepsilon}(A) \leq n\}.
\]
If $A = B_{\|\cdot\|}^d = \{ x\in \R^d: \|x\| \leq 1\}$ is a unit ball, we also use the notation
\[
  \varepsilon_n(\id: \ell_{\|\cdot\|}^d \to \ell_\infty^d) := \varepsilon_n(B_{\|\cdot\|}^d,\ell_\infty^d).
\]
In the applications which we have in mind $\|\cdot\|$ will be a classical (quasi-)norm $\|\cdot\| = \|\cdot\|_p$ for
$0<p\leq \infty$. In this case, the behavior in $n$ and $d$ of the entropy numbers $\varepsilon_n(\id: \ell_p^d \to
\ell_\infty^d)$ is completely understood \cite{EdTr96, Ku01, RS11, Sch84, Tr97}. For the reader's convenience, we
restate the results.

\begin{proposition}\label{res:entropy_infty}
For all $n \in \N$, we have
\[ 
 n^{-1/d} \leq \varepsilon_n(id: \ell_{\infty}^d \to \ell_{\infty}^d) \leq 2 n^{-1/d}.
\]
However, $\varepsilon_n(id: \ell_{\infty}^d \to \ell_{\infty}^d) = 1$ as long as $n < 2^d$.
\end{proposition}

\begin{proposition}\label{res:entropy_p} Let $0 < p < \infty$. Then,
$$
\varepsilon_n(id: \ell_p^d \to \ell_{\infty}^d) \asymp \left\{
\begin{array}{rcl}1&:&1\leq n \leq d,\\
\Big(\frac{\log(1+d/\log n)}{\log n}\Big)^{1/p}&:& d \leq n \leq 2^d,\\
d^{-1/p} n^{-1/d}&:& n \geq 2^d\,,
\end{array}
\right.
$$
with constants independent of $n$ and $d$.
\end{proposition}

\noindent The equivalence constants in Proposition \ref{res:entropy_p} are not further specified in the literature. It
is possible to calculate explicit, but rather lengthy expressions. We refrain from going more into detail at this point.
In Section \ref{sec:entropy}, we will comment on the behavior of the constants for $n \to \infty$.

\begin{remark}
The closely related entropy numbers $\varepsilon_n(\Sphere_p, \ell_\infty^d)$, where $\Sphere_p = \{ x \in \R^d: \|x\|_p
= 1\}$, have been understood only lately \cite{HM15, MUV13}. It is no surprise that these behave identically to the
entropy numbers $\varepsilon_n(id: \ell_p^d \to \ell_{\infty}^d)$, except that asymptotically they decay as
$n^{-1/(d-1)}$. To prove the bounds on $\varepsilon_n(\Sphere_p, \ell_\infty^d)$ one largely mimics the well-known proof
for $\varepsilon_n(id: \ell_p^d \to \ell_{\infty}^d)$. Surprisingly, there is one case where this strategy fails. For
$0<p<1$ and $n \geq 2^d$ the familiar volume arguments become inaccurate and it needs different techniques to obtain
matching bounds, see \cite{HM15}.
\end{remark}

\paragraph*{Notions of tractability} In the course of this paper we want to classify how the dimension $d$ affects the hardness of the approximation problem $\Id: H^{\mathbf{w}}(\T^d) \to L_2(\T^d)$ depending on the weight sequence $\mathbf{w}$. The field of information-based complexity provides notions of tractability \cite{NW1}, which rate the difficulty of the approximation problem in terms of how its information
complexity
\[
 n(\varepsilon,d) := \inf \{ n \in \N: a_n (\Id:H^{\mathbf{w}}(\T^d) \to L_2(\T^d)) \leq \varepsilon \}
\]
grows in $1/\varepsilon$ and $d$. Let us first note that for all weight sequences $\mathbf{w}$ considered in this paper
we have an initial error
\[
 a_1(\Id:H^{\mathbf{w}}(\T^d) \to L_2(\T^d)) = \|\Id:H^{\mathbf{w}}(\T^d) \to L_2(\T^d)\| = 1.
\]
Hence, the normalized (relative) error and the absolute error coincide.
Now, the approximation problem is said to be \emph{polynomially tractable} if $n(\varepsilon,d)$ is bounded polynomially in $\varepsilon^{-1}$ and $d$, i.e., there exist numbers $C,r,q > 0$ such that 
\[
 n(\varepsilon,d) \leq C \, \varepsilon^{-r} \, d^q \mbox{ for all $0<\varepsilon<1$ and all $d\in{\N}$.}
\]
The approximation problem is called 
\emph{quasi-polynomially tractable} if there exist two constants $C,t>0$ such that 
\[
    n(\varepsilon,d) \leq C\exp(t(1+\ln(1/\varepsilon))(1+\ln d))\,.
\]
It is called \emph{weakly tractable} if
\be\label{eq:wtrac}
    \lim\limits_{1/\varepsilon+d\to\infty} \frac{\log n(\varepsilon,d)}{1/\varepsilon + d} = 0\,,
\ee
i.e., the information complexity $n(\varepsilon,d)$ neither depends exponentially on $1/\varepsilon$ nor on $d$. We say
that the approximation problem is \emph{intractable}, if \eqref{eq:wtrac} does not hold. If for some fixed
$0<\varepsilon<1$ the information complexity $n(\varepsilon,d)$ 
is an exponential function in $d$ then we say that the problem suffers from
{\em the curse of dimensionality}. To make it precise, we face the curse if there exist positive numbers $c,
\varepsilon_0, \gamma$ such that
\begin{align*}
 n(\varepsilon,d) \geq c(1+\gamma)^d\,,\quad \mbox{for all } 0<\varepsilon\leq \varepsilon_0 \mbox{ and infinitely many
}d\in \N\,.
\end{align*}

\section{Counting via entropy}\label{sec:entropy}
The \emph{grid number} $G(A)$ of a set $A \subseteq \R^d$ is the number of points in $A$ that lie on the grid $\Z^d$.
Formally,
\[ 
 G(A) = \sharp ( A \cap \Z^d).
\]
The grid numbers $G(r B_{\|\cdot\|}^d)$, $r \in \R$, are central in the study of approximation numbers $a_n(\Id:
H^{\mathbf{w}}(\T^d) \to L_2(\T^d))$ if the weights $\mathbf{w}$ are induced by some (quasi-)norm $\|\cdot\|$, see
Section \ref{sec:approximation_numbers} below. In this section, we show that the combinatorics for grid numbers can be
reduced to covering arguments, at least if the studied set is solid. We call a set $A \subseteq \R^d$ \emph{solid} if
for all $x \in A$ every vector $y \in \R^d$ which component-wise fulfills $|y_i| \leq |x_i|$ is contained in $A$. For
instance, the unit ball $B_p^d$ is solid for any $0<p\leq \infty$.
\begin{lemma}\label{res:grid_numbers}
For a solid set $A \subseteq R^d$ we have
 \[ 
  N_1(A) \leq G(A) \leq N_{\rho}(A)
 \]
for any $\rho < 1/2$.
\end{lemma}

\begin{proof}
For $x \in A$, we define $\lfloor x \rfloor$ component-wise by $\lfloor x \rfloor_j := \sign x_j \; \lfloor |x_j|
\rfloor$. Clearly, $\|\lfloor x \rfloor-x\|_{\infty} < 1$ for any $x \in A$. Since the set $A$ is solid, $x \in A$ implies $\lfloor x \rfloor \in A$. Hence, the
intersection $A \cap  \Z^d$ forms a $1$-net of $A$ in $\ell_{\infty}^d$. Consequently, we have $N_1(A) \leq G(A)$. The
upper bound is a direct consequence of the fact that it needs at least $G(A)$ many balls of radius $\rho < 1/2$ to cover
$A \cap \Z^d$.
\end{proof}

\noindent A function $\|\cdot\|: \R^d \to [0,\infty)$ is called a $p$-norm for some $0<p\leq 1$ if $\|\cdot\|$ fulfills the norm axioms of absolute homogeneity and point separation and, furthermore, the $p$-triangle inequality
\[ 
 \|x + y\|^p \leq \|x\|^p + \|y\|^p 
\]
holds true for any $x,y \in \R^d$.
The typical example for a $p$-norm with $0<p<1$ is $\|\cdot\| = \|\cdot\|_p$. If $A \subset \R^d$ is the unit ball of a $p$-norm $\|\cdot\|$,
there is another relation
between covering and grid numbers. In this relation the quantity
\[ 
 \lambda_{\|\cdot\|}(d) := \big\|\sum_{i=1}^d e_i\big\|
\]
appears, where $e_1,\dots,e_d$ denote the canonical basis vectors in $\R^d$. Note that if $\|\cdot\| = \|\cdot\|_p$ for
$0<p\leq \infty$ we have $\lambda_{\|\cdot\|}(d) = d^{1/p}$.

\begin{lemma}\label{res:grid_numbers_p_balls}
Let $\|\cdot\|$ be a $p$-norm in $\R^d$ for some $0<p\leq 1$. For $r > \lambda_{\|\cdot\|}(d)/2$, put
  \begin{align*}\label{eq:rescaled_radii}
  \begin{split}
  l(r,p,d) &:=\big(r^{p} - \lambda_{\|\cdot\|}(d)^{p}/2^{p}\big)^{1/p},\\
  L(r,p,d) &:= \big(r^{p} + \lambda_{\|\cdot\|}(d)^{p}/2^{p}\big)^{1/p}.
  \end{split}
  \end{align*} 
With $B_{\|\cdot\|}^d$ denoting the unit ball, we have the relation
 \[ 
  N_{1/2}(l(r,p,d) B_{\|\cdot\|}^d) \leq G(r B_{\|\cdot\|}^d) \leq N_{1/2}(L(r,p,d) B_{\|\cdot\|}^d).
 \]
\end{lemma}
\begin{proof}
Let $Q_k = k + [-1/2,1/2]^d$. The $p$-triangle inequality for $\|\cdot \|$ yields
\begin{align}\label{eq:two_sided_inclusion}
 l(r,p,d) B_{\|\cdot\|}^d  \subseteq \bigcup_{\substack{k \in \Z^d,\\\|k\| \leq r}} Q_k \subseteq L(r,p,d)
B_{\|\cdot\|}^d.
\end{align}
The left-hand side inclusion shows that the set $r B_{\|\cdot\|}^d \cap  \Z^d$ is a $1/2$-net of $l(r,p,d)
B_{\|\cdot\|}^d$ in $\ell_{\infty}^d$. This shows the left-hand side inequality of the statement. The second inequality
follows from the right-hand side inclusion by a simple volume argument.
\end{proof}

Lemma \ref{res:grid_numbers_p_balls} yields the following bounds for entropy numbers. Note that the upper bound is a
refinement for large $n$ of the usual upper bound found in the literature, compare also with Propositions
\ref{res:entropy_infty}, \ref{res:entropy_p}.

\begin{lemma}\label{res:entropy_numbers_large_n}
Let $\|\cdot\|$ be a $p$-norm for some $0<p\leq 1$. For $n > (d^{1/p}/2)^d \vol(B_{\|\cdot\|}^d)$, we have
\[ 
 \frac{1}{2} (n/\vol(B_{\|\cdot\|}^d))^{1/d}  \leq \varepsilon_n(\id: \ell_{\|\cdot\|}^d \to \ell_\infty^d) \leq \frac{1}{2}
\big( (n/\vol(B_{\|\cdot\|}^d))^{p/d} - 2^{1-p} d^{1 /p})\big)^{-1/p}.
\]
\end{lemma}
\begin{proof}
The lower bound is the standard lower bound, which follows from simple volume arguments and in fact holds true for all
$n \in \N$. To see the upper bound, choose $r = \big( (n/\vol(B_{\|\cdot\|}^d))^{p/d} - d^{p /p}/2^{p})\big)^{1/p}$.
Then it follows from the right-hand side inclusion of \eqref{eq:two_sided_inclusion} that $G(r B_p^d) \leq n$ and
further from the left-hand side inclusion of \eqref{eq:two_sided_inclusion} that $\varepsilon_n(\id: \ell_p^d \to
\ell_\infty^d) \leq 1/(2l(r,p,d))$, where $l(r,p,d)$ is defined in Lemma \ref{res:grid_numbers_p_balls}. It remains to
plug in the formula for $r$. 
\end{proof}

Let us briefly come back to the discussion on the equivalence constants in Proposition \ref{res:entropy_infty} and
Proposition \ref{res:entropy_p}. An interesting question is whether the equivalence constants in the lower and upper
bounds necessarily have to be different or whether this is just an artifact of the used proof techniques. Lemma
\ref{res:entropy_numbers_large_n} allows a partial answer. In the limit $n \to \infty$ we have
\begin{align}\label{eq:entropy_limit}
  \lim_{n \to \infty} n^{1/d} \varepsilon_n(id: \ell_{\|\cdot\|}^d \to \ell_{\infty}^d) = 1/2 \cdot
\vol(B_{\|\cdot\|}^d)^{1/d}.
\end{align}
If $\|\cdot\| = \|\cdot\|_p$ for $0<p\leq \infty$, then $\vol(B_{\|\cdot\|}^d)^{1/d} = \vol(B_p^d)^{1/d} \asymp
d^{-1/p}$, see \cite{W05}. 



\section{Characterization of approximation numbers}
\label{sec:approximation_numbers}

In this section, we prove a number of characterization results for approximation numbers $a_n(\Id: H^{\mathbf{w}}(\T^d)
\to L_2(\T^d))$ when the weight sequence $\mathbf{w}$ is derived from some (quasi-)norm in $\R^d$. To begin with, let us
recapitulate some well-known facts about approximation numbers of weighted spaces. Let $\mathbf{w} = (w(k))_{k \in \Z^d}$ be an arbitrary weight sequence such that $1/\mathbf{w} := (1/w(k))_{k \in \Z^d} \in \ell_\infty(\Z^d)$. It is well-known that the approximation numbers are given by the \emph{non-increasing rearrangement}
$(\sigma_n)_{n \in \N}$ of the inverse weight sequence $1/\mathbf{w}$, that is,
\begin{align}
\label{eq:appr_nr_rearr}
 a_n(\Id: H^{\mathbf{w}}(\T^d) \to L_2(\T^d)) = \sigma_n
\end{align}
for all $n \in \N$. We briefly sketch the proof of this fact, details and further references can be found in \cite[Section
2.2]{KSU13}. Consider the isometries
\begin{align}\label{eq:seq_op}
 A^{\textbf{w}}: H^{\mathbf{w}}(\T^d) \to \ell_2(\Z^d), \; f \mapsto \big(w(k) c_k(f) \big)_{k \in \Z^d}
\end{align}
and
\begin{align}\label{eq:fourier_op}
 F: \ell_2(\Z^d) \to L_2(\T^d), \; (\xi_k)_{k \in \Z^d} \mapsto (2\pi)^{-d/2} \sum_{k \in \Z^d} \xi_k e^{i k x},
\end{align}
as well as the diagonal operator
\begin{align}\label{eq:diag_op}
 D^{\textbf{w}}: \ell_2(\Z^d) \to \ell_2(\Z^d), \; (\xi_k)_{k \in \Z^d} \mapsto ( \xi_k/w(k))_{k \in \Z^d}.
\end{align}
Obviously, we have $\Id = F \circ D^{\textbf{w}} \circ A^{\textbf{w}}$, which is illustrated by the commutative diagram
below.
\begin{figure}[h]
\centering
\begin{tikzpicture}
  \tikzset{node distance=3cm, auto}
  \node (H) {$H^{\mathbf{w}}(\T^d)$};
  \node (L) [right of =H] {$L_2(\T^d)$};
  \node (ell) [below of=H] {$\ell_2(\Z^d)$};
  \node (ell2) [right of=ell] {$\ell_2(\Z^d)$};
  \draw[->] (H) to node {$\Id$} (L);
  \draw[->] (H) to node {$A^{\textbf{w}}$} (ell);
  \draw[->] (ell) to node {$D^{\textbf{w}}$} (ell2);
  \draw[->] (ell2) to node {$F$} (L);
\end{tikzpicture}
\caption{Commutative diagram for the embedding $\Id: H^{\mathbf{w}}(\T^d) \to L_2(\T^d)$.}
\label{fig:commutative_diagram}
\end{figure}

\noindent It is known that the approximation numbers of the diagonal operator are given by $(\sigma_n)_{n \in \N}$, and
from $\|A^{\textbf{w}}\| = \|F\| = 1$ we conclude
\begin{align*}
 a_n(\Id: H^{\mathbf{\mathbf{w}}}(\T^d) \to L_2(\T^d)) = a_n(D^{\textbf{w}}: \ell_2(\Z^d) \to \ell_2(\Z^d)) = \sigma_n.
\end{align*}

We come to our first characterization result. For weight sequences $\mathbf{w}$ given by a (quasi-)norm $\|\cdot\|$ on
$\R^d$, we show that the non-increasing rearrangement $(\sigma_n)_{n\in \N}$ is in fact equivalent up to constants to
the entropy numbers $\varepsilon_n(\id: \ell_{\|\cdot\|}^d \to \ell_\infty^d)$.

\begin{theorem}\label{res:approx_entropy}
Let $\|\cdot\|$ be some (quasi-)norm on $\R^d$ such that $\min_{i=1,\dots,d} \|e_i\| = 1$, where $e_1,\dots,e_d$ denotes
the canonical basis in $\R^d$. Consider the weight sequence $\mathbf{w} =(w(k))_{k \in \Z^d}$ given by $w(k) :=
\max\{1,\|k\|\}$. For every $n \in \N$, we have
\[ 
 1/2 \, \varepsilon_{n}(\id: \ell_{\|\cdot\|}^d \to \ell_{\infty}^d)
 \leq
 a_n(\Id: H^{\mathbf{w}}(\T^d)  \to L_2(\T^d))
 \leq
 4 \, \varepsilon_{n}(\id: \ell_{\|\cdot\|}^d \to \ell_{\infty}^d).
\]
\end{theorem}

\begin{proof}
Let $(\sigma_n)_{n \in \N}$ denote the non-increasing rearrangement of $(1/w(k))_{k \in \Z^d}$. By
\eqref{eq:appr_nr_rearr} we know that $a_n(\Id: H^{\mathbf{w}}(\T^d)  \to L_2(\T^d)) = \sigma_n$. Since $G(m
B_{\|\cdot\|}^d) = \sharp \{ k \in \Z^d: \|k\| \leq m \} = \sharp \{k \in \Z^d: w(k) \leq m \}$ and $w(me_{i^*}) = m$,
where $i^{*} = \arg\min_{i=1,\dots,d}\|e_i\|$, we  have $\sigma_{G(mB_{\|\cdot\|}^d)} = 1/m$.

Let us first prove the upper bound. For brevity, we write $\varepsilon_n = \varepsilon_n(id: \ell_{\|\cdot\|}^d \to
\ell_{\infty}^d)$ in the following.  For given $n  \in \N$, let $\varepsilon > \varepsilon_n$ and put $m := \lfloor
1/((2+\delta)\varepsilon) \rfloor$ for some $\delta > 0$. By virtue of Lemma \ref{res:grid_numbers}, we obtain
\[ 
n \geq N_{\varepsilon}(B_{\|\cdot\|}^d) = N_{m\varepsilon}(m B_{\|\cdot\|}^d) \geq N_{1/(2+\delta)}(m B_{\|\cdot\|}^d)
\geq G(m B_{\|\cdot\|}^d).
\]
The monotonicity of approximation numbers yields
\begin{align}\label{eq:sigma_n_estimate}
 \sigma_n \leq \sigma_{G(m B_{\|\cdot\|}^d)} = 1/m \leq 2(2+\delta)\varepsilon.
\end{align}
Since $\delta$ can be chosen arbitrarily close to $0$ and $\varepsilon$ arbitrarily close to $\varepsilon_n$, we reach
at $\sigma_n \leq 4 \varepsilon_n$.

To prove the lower bound, assume $\varepsilon < \varepsilon_n$ for some $n \in \N$ and put $m = \lceil 1/\varepsilon
\rceil$. We have
$n \leq N_{\varepsilon}(B_{\|\cdot\|}^d) \leq N_{1}(m B_{\|\cdot\|}^d) \leq G(m B_{\|\cdot\|}^d),$
where the last estimate is due to Lemma \ref{res:grid_numbers}. Thus,
$\sigma_n \geq \sigma_{G(m B_{\|\cdot\|}^d)} = 1/m \geq 1/2 \varepsilon.$
Again, as $\varepsilon$ may be chosen arbitrarily close to $\varepsilon_n$, we have
$ \sigma_n \geq 1/2 \varepsilon_n.$
\end{proof}

\begin{remark} {\em (i)} The constants $1/2$ and $4$ are an artifact of the proof technique. We do not claim that these are optimal.\\
{\em (ii)} The assumption $\min_{i=1,\dots,d} \|e_i\| = 1$ in Theorem \ref{res:approx_entropy} has only been made to keep the
formulation of the statement and the proof as simple as possible. In particular, the initial error is always $1$. If $\min_{i=1,\dots,d} \|e_i\| = c \neq 1$, then the statement still holds true, provided we define
$w(k) := \max\{c,\|k\|\}$. Otherwise, i.e. when keeping the definition $w(k) := \max\{1,\|k\|\}$, the statement holds
true for sufficiently large $n > n_0(c,d)$, where $n_0(c,d)$ can depend on $c$ and the dimension $d$.
\end{remark}

\noindent The statement of Theorem \ref{res:approx_entropy} can be easily generalized.

\begin{theorem}\label{res:approx_entropy_mono}
Let $\|\cdot\|$ be some (quasi-)norm on $\R^d$ as in Theorem \ref{res:approx_entropy} and let $\varphi: \R \to \R$ be a
monotonically increasing function satisfying $\varphi(0) = 1$. Consider the weight sequence $\mathbf{w} =
\varphi(\|\cdot\|)$ given by $w(k) := \varphi(\|k\|)$. Writing
$ 
 \varepsilon_n = \varepsilon_n(\id: \ell_{\|\cdot\|}^d \to \ell_\infty^d),
$ we have, for all $n \in \N \setminus \{1\}$, the estimate
\[ 
 \frac{1}{\varphi(2/\varepsilon_n)}
 \leq
 a_n(\Id: H^{\varphi(\|\cdot\|)}(\T^d) \to L_2(\T^d))
 \leq
 \frac{1}{\varphi(1/(4\varepsilon_n))}.
\]
\end{theorem}

\begin{proof}
Let $\widetilde{\mathbf w}$ be the weight sequence given by $\widetilde{w}(k) = \max\{1,\|k\|\}$ for $k \in \Z^d$.
Further, let $(\sigma_n)_{n \in \N}$ be the non-increasing rearrangement of $1/\widetilde{\mathbf{w}}$. Note that
$\varphi(\|k\|)= \varphi(\widetilde{w}(k))$ for $k \neq 0$ since $\min_{i=1,\dots,d} \|e_i\|=1$. Put $\gamma_1 = 1$ and
\begin{align}
\label{eq:rearr_mon} 
 \gamma_n = \frac{1}{\varphi(1/\sigma_n)}
\end{align}
for natural $n > 1$. Since $\varphi$ is monotonically increasing the sequence $(\gamma_n)_{n \in \N}$ is non-increasing
and thus the non-increasing rearrangement of $(1/w(k))_{k \in \Z^d}$. It remains to combine \eqref{eq:rearr_mon} with
the finding of Theorem \ref{res:approx_entropy}. 
\end{proof}

The constants in the lower and upper bound of Theorem \ref{res:approx_entropy} do not match. A consequence of Theorem \ref{res:general_limit} below is the limit result
\[ 
 \lim_{n \to \infty} \frac{a_n(\Id: H^{\max\{1,\|\cdot\|\}}(\T^d)  \to L_2(\T^d))}{\varepsilon_{n}(\id:
\ell_{\|\cdot\|}^d \to \ell_{\infty}^d)} = 2,
\]
which suggests that the true constant in the lower and upper bound should be $2$ for sufficiently large $n$. 

\begin{theorem}\label{res:general_limit}
Let $\|\cdot\|$ be a $p$-norm in $\R^d$ for $0<p\leq 1$. Further,  let $\varphi$ be given by $\varphi(t) =
\exp(\beta{g(t))}$ with $\beta > 0$ and monotonically increasing, differentiable $g$ satisfying $g(0) = 0$ and $\lim_{t
\to \infty} g'(t) t^{1-p}= 0$. Recall the weight sequence $\mathbf{w} = \varphi(\|\cdot\|)$ defined in Theorem
\ref{res:approx_entropy_mono}. Using the shorthands
$
 a_n = a_n(\Id: H^{\varphi(\|\cdot\|)}(\T^d) \to L_2(\T^d))
$
and
$ 
 \varepsilon_n = \varepsilon_n(\id: \ell_{\|\cdot\|}^d \to \ell_\infty^d),
$
it holds true that
\[
 \lim_{n \to \infty} a_n
 \varphi(1/(2\varepsilon_n))= 1.
\]
\end{theorem}

\begin{proof}
Let $(\sigma_n)_{n \in \N}$ be the non-increasing rearrangement of $(1/\max\{1,\|k\|\})_{k \in \Z^d}$ and $(\gamma_n)_{n
\in \N}$ be the non-increasing rearrangement of $(1/\varphi(\|k\|))_{k \in \Z^d}$. Let us first refine the upper bound
\eqref{eq:sigma_n_estimate}. Consider $n \in \N$ sufficiently large such that $\varepsilon_n < 1/2$. With $\delta > 0$
arbitrary, $\varepsilon > 0$ such that $\varepsilon_n < \varepsilon < 1/(2+\delta)$, and $m := \lfloor
1/((2+\delta)\varepsilon) \rfloor$, we obtain
\[ 
 \sigma_n \leq \sigma_{G(m B_p^d)} = 1/m \leq  \frac{(2+\delta)\varepsilon}{1-(2+\delta)\varepsilon}.
\]
Since we may choose $\delta$ arbitrarily close to $0$ and $\varepsilon$ arbitrarily close to $\varepsilon_n$, we obtain
$\sigma_n \leq 1/h_1(1/(2\varepsilon_n))$, where $h_1(t) = t-1$.

To obtain a refinement of the lower bound, choose for $\varepsilon < \varepsilon_n$ the natural number $m = \lceil
1/(2\varepsilon)\rceil$.
Using Lemma \ref{res:grid_numbers_p_balls} we obtain $n \leq N_\varepsilon(B_{\|\cdot\|}^d) \leq N_{1/2}(m
B_{\|\cdot\|}^d) \leq G(\widetilde m B_{\|\cdot\|}^d)$, where
\[
 \widetilde m := \big(m^{p} + \lambda_{\|\cdot\|}(d)^{p}/2^{p}\big)^{1/p}.
\]
Hence,
\[ 
 \sigma_n \geq 1/\widetilde{m} \geq 1/h_2(1/(2\varepsilon)),
\]
where $h_2(t) = ((t+1)^p+(\lambda_{\|\cdot\|}(d)^p/2^p)^{1/p}$.
Since we may choose $\varepsilon$ arbitrarily close to $\varepsilon_n$, we obtain $1/h_2(1/(2\varepsilon_n)) \leq
\sigma_n$.
   
Combining the refined estimates with equation \eqref{eq:rearr_mon}, we obtain from multiplying by
$\varphi(1/(2\varepsilon_n))$ the two-sided estimate
\begin{align}\label{eq:general_limit_two_sided_estimate} 
\frac{\varphi(1/(2\varepsilon_n))}{\varphi(h_2(1/(2\varepsilon_n)))}
\leq \gamma_n \varphi(1/(2\varepsilon_n)) \leq
\frac{\varphi(1/(2\varepsilon_n))}{\varphi(h_1(1/(2\varepsilon_n)))}.
\end{align}
We have $\ln\left(\frac{\varphi(x)}{\varphi(h_1(x))}\right) = \beta(g(x) - g(x-1)) = \beta g'(\xi_x)$ for some $x-1 \leq
\xi_x \leq x$. From the assumptions on $g$, it obviously follows that $\lim_{x \to \infty}g'(\xi_x) = 0$. Hence, we have
$\lim_{n \to \infty}\gamma_n \varphi(1/(2\varepsilon_n)) \leq 1$. For the estimate from below we have to show that
$g(x) - g(h_2(x)) \to 0$ for $x \to \infty$. By the mean value theorem, it follows that
\begin{align*}
|x - h_2(x)| &\leq 1 + |x+1-h_2(x)|\\
& \leq 1 +\frac{\lambda_{\|\cdot\|}(d)^{p}}{p2^p} [(x+1)+ \mu]^{1/p-1}
\end{align*} 
for some $\mu \in [0, \lambda_{\|\cdot\|}(d)^{p}/2^{p}]$. Due to $p \leq 1$ we further may estimate $|x - h_2(x)| \leq
C_d x^{1-p}$ for some $C_d > 0$. Combined with another application of the mean value theorem, this yields
\[
 |g(x) - g(h_2(x))| \leq |g'(\xi)||x - h_2(x)| \leq C_d |g'(\xi)| x^{1-p} \leq C_d|g'(\xi)| \xi^{1-p},
\]
where $\xi \in [x,h_2(x)]$. Since we have assumed $\lim_{x \to \infty} g'(x) x^{1-p} = 0$ it follows that $1 \leq \lim_{n
\to \infty} \gamma_n \varphi(1/(2\varepsilon_n))$.
\end{proof}

\section{Isotropic Sobolev spaces}\label{sec:isotropic_sobolev}

In this section, we give further details and additional remarks to the results presented in Subsection
\ref{sec:intro:isotropic} of the introduction.

\begin{proof}[Proof of Theorem \ref{res:approx_numbers}]
For $0<p<\infty$, Theorem \ref{res:approx_entropy_mono} with $\|\cdot\| = \|\cdot\|_p$ and $\varphi(t) = (1 +
t^p)^{s/p}$ yields
\[ 
 2^{-(1+s/p)} \varepsilon_n(\id: \ell_p^d \to \ell_\infty^d)^s \leq a_n(\Id: H^{s,p}(\T^d) \to L_2(\T^d)) \leq 4^s
\varepsilon_n(\id: \ell_p^d \to \ell_\infty^d)^s.
\]
It remains to apply Proposition \ref{res:entropy_p}. In case $p = \infty$, the argumentation is analogous with
$\varphi(t) = \max\{1,t\}^s$.
\end{proof}
\noindent In the special case $p = \infty$, let us restate Theorem \ref{res:approx_numbers} with explicit expressions
for the equivalence constants.
\begin{theorem}\label{res:approx_numbers_p_infty}
For $p=\infty$ and $s> 0$, we have
\[ 
 2^{-(1+1/p)s} n^{-s/d} \leq a_n (\Id:H^{s,\infty}(\T^d) \to L_2(\T^d)) \leq 8^s n^{-s/d}.
\]
for all $n \in \N$. However, $a_n (\Id:H^{s,\infty}(\T^d) \to L_2(\T^d)) = 1$ as long as $n \leq 2^d$. 
\end{theorem}
\begin{proof}
Combine Theorem \ref{res:approx_entropy_mono} with $\varphi(t) = \max\{1,t\}^s$ and Proposition \ref{res:entropy_infty}.
\end{proof}

\noindent Theorem \ref{res:general_limit} applied to the approximation numbers $a_n(\Id: H^{s,p}(\T^d) \to L_2(\T^d))$
yields the following corollary.
\begin{corollary}\label{res:iso_asymp_constant} Let $0<p\leq \infty$ and $s>0$. Then 
\[ 
 \lim_{n \to \infty} n^{s/d} a_n(\Id: H^{s,p}(\T^d) \to L_2(\T^d)) = (\vol(B_p^d))^{s/d} \asymp d^{-s/p}.
\]
\end{corollary}

\begin{remark}
This work continues the considerations made in \cite{KSU13}. Theorems
\ref{res:approx_numbers}, \ref{res:approx_entropy}, and \ref{res:approx_numbers_p_infty} extend \cite[Thm. 4.3, 4.11,
and 4.14]{KSU13}, which covered only the cases $p=1$, $p=2$, and $p=2s$. Moreover, by
Theorem \ref{res:approx_numbers} we close the logarithmic gap in \cite[Thm. 4.6]{KSU13} and confirm \cite[Rem.
4.7]{KSU13}. 
\end{remark}

\begin{remark}[Approximation in $L_\infty$]
In our results the approximation error is measured in $L_2(\T^d)$. It would be highly interesting to have analogs to
Theorem \ref{res:approx_numbers}, Theorem \ref{res:approx_entropy}, and Theorem \ref{res:approx_numbers_p_infty} for the
approximation in $L_\infty(\T^d)$. The recent work \cite{CKS14} provides the formula
\[ 
 a_n(\Id: H^{s,p}(\T^d) \to L_\infty(\T^d)) = (\sum_{j \geq n} \sigma_j^2)^{1/2},
\]
where $(\sigma_n)_{n \in \N}$ is again the non-increasing rearrangement of the inverse weight sequence
$(1/w_{s,p}(k))_{k \in \Z^d}$. In principle, this allows to prove an analog to Theorem \ref{res:approx_entropy}.
However, the constants in the known bounds for the entropy numbers of the embbeding $\id: \ell_p^d \to \ell_\infty^d$
are not good enough to obtain meaningful preasymptotics.
\end{remark}

%


\section{Spaces of Gevrey type}\label{sec:gevrey}
In this section, we study approximation numbers of spaces $G^{\alpha,\beta,p}(\T^d) =
H^{\mathbf{w}^G_{\alpha,\beta,p}}(\T^d)$ with exponential weights given by
\begin{align}\label{eq:gevrey_weights}
 w^G_{\alpha,\beta,p}(k) = \exp(\beta \|k\|_p^\alpha), \quad k \in \Z^d.
\end{align}
As already indicated in the introduction, the study of spaces $G^{\alpha,\beta,p}(\T^d)$ is motivated by classical
Gevrey classes. Let us elaborate a bit more on this before we discuss our results in detail. For the interested reader
we note that a standard reference on Gevrey spaces and its applications is Rodino's book \cite{R93}.

The classical Gevrey class ${\bf G}^\sigma(\mathbb{R}^d)$, $\sigma>1$, consists of all
$f\in C^\infty(\mathbb{R}^d)$ with the following property:
\begin{quote}
For every compact subset $K\subset\R^d$ there are constants 
$C,R>0$ such that for all $x\in K$ and all multi-indices $\alpha=(\alpha_1,\ldots,\alpha_d)\in\N_0^d$ the inequality 
$$
|D^\alpha f(x)|\leq C R^{\alpha_1+\ldots+\alpha_d}(\alpha_1!\cdots\alpha_d!)^s
$$
holds.
\end{quote}
If $f$ is $2\pi$-periodic in each coordinate, i.e. if $f\in C^\infty(\mathbb{T}^d)$, the growth 
conditions on the derivatives can be rephrased in terms of Fourier coefficients: $f$ belongs to 
${\bf G}^\sigma(\mathbb{R}^d)$ if and only if there exists a constant $\beta>0$ such that
$$
\sum_{k\in\Z^d}\exp(2\beta\,\Vert k \Vert_1^{1/\sigma})|c_k(f)|^2<\infty\,.
$$
Here one can replace $\Vert k\Vert_1$ by any other (quasi-)norm on $\R^d$. This gives only a different constant $\beta$,
but the 
exponent $1/\sigma$ does not change. This was the motivation in \cite{KP15} to introduce the periodic Gevrey spaces 
$G^{\alpha,\beta,p}(\T^d)$, $0<\alpha<1$, $0<\beta,p<\infty$, which
consist of all $f\in C^\infty(\mathbb{T}^d)$ such that the norm
$$
\Vert f |G^{\alpha,\beta,p}(\T^d)\Vert:=\Big( \sum_{k\in\Z^d}\exp(2\beta\,\Vert k
\Vert_p^\alpha)|c_k(f)|^2\Big)^{1/2}
$$
is finite. For convenience of notation we changed the exponent, setting $\alpha:=1/\sigma$. Clearly, all these spaces
are 
Hilbert spaces.

In the definition of $G^{\alpha,\beta,p}(\mathbb{T}^d)$ one can extend the range 
of parameters to $\alpha>0$. The decisive difference is that the periodic Gevrey spaces, i.e. those with $0<\alpha<1$,
contain
non-analytic functions, while for $\alpha\ge 1$ all functions in $G^{\alpha,\beta,p}(\mathbb{T}^d)$ are analytic.

We come to the first result of this section. As an immediate consequence of Propositions \ref{res:entropy_infty},
\ref{res:entropy_p} and Theorem \ref{res:approx_entropy_mono} we obtain
\begin{theorem}\label{res:gevrey}
Let $\alpha,\beta > 0$ and $0<p\leq \infty$. Consider the approximation numbers
\[ 
 a_n := a_n(\Id: G^{\alpha,\beta,p}(\T^d) \to L_2(\T^d)).
\]
\begin{enumerate}[label=(\roman*)]
 \item For $1 \leq n \leq d$, we have $a_n \asymp_{\alpha,\beta,p} 1$.
 \item For $d \leq n \leq 2^d$, we have
 \[ 
  -\ln(a_n) \asymp_{\alpha,p} \beta \left[\frac{\log(n)}{\log(1+d/\log(n))}\right]^{\alpha/p}.
 \]
 \item For $n \geq 2^d$, we have
 \[ 
  -\ln(a_n) \asymp_{\alpha,p} \beta d^{\alpha/p} n^{\alpha/d}.
 \]
\end{enumerate}
\end{theorem}

The limit result in Theorem \ref{res:general_limit} can be specialized as follows. Unfortunately, our proof technique
does not work for classes of analytic functions.

\begin{theorem}\label{limit:gevrey}
Let $0 < p \leq \infty$, $0< \alpha < \min\{1,p\}$, and $\beta > 0$. For
\[ 
 a_n := a_n(\Id: G^{\alpha,\beta,p}(\T^d) \to L_2(\T^d)),
\]
we have
\[ 
 \lim_{n \to \infty} a_n \cdot \exp(\beta \vol(B_p^d)^{-\alpha/d} n^{\alpha/d}) = 1.
\]
\end{theorem}
\begin{proof}
Let $\tilde p := \min \{1,p\}$. Further, let $h_2(x) = ((x+1)^{\tilde p}+d^{\tilde p/p}/2^{\tilde p})^{1/\tilde{p}}$ and
$h_3(x) = (x^{\tilde p} - d^{\tilde p/p}/2^{1-\tilde p})^{1/\tilde{p}}-1$. If we put $x_n = n^{1/d} \vol(B_p^d)^{-1/d}$,
then the two-sided estimate \eqref{eq:general_limit_two_sided_estimate} in combination with Lemma
\ref{res:entropy_numbers_large_n} can be reformulated as
\[ 
 \exp(\beta (x_n^\alpha - h_2(x_n)^\alpha)) \leq a_n \exp(\beta \vol(B_p^d)^{-\alpha/d} n^{\alpha/d}) \leq \exp(\beta(x_n^\alpha -
h_3(x_n)^\alpha)).
\]
Copying the arguments given below of \eqref{eq:general_limit_two_sided_estimate} we conclude that $\lim_{n \to \infty}
(x_n^\alpha - h_2(x_n)^\alpha) = 0$ if $\alpha < \tilde p$. Using similar arguments, we also get that $\lim_{n \to
\infty} (x_n^\alpha - h_3(x_n)^\alpha) = 0$ if $\alpha < \tilde p$.
\end{proof}

\subsection{A connection with spaces of dominating mixed smoothness}\label{sec:preasymptotic_similarity}

In the special situation $\alpha = p$, the estimate in Theorem \ref{res:gevrey} (ii) can be written more transparently.
Namely, for $d \leq n \leq 2^d$, we have constants $c_1(p)$ and $c_2(p)$ such that 
\begin{align}\label{eq:gevrey_quasi_polynomial_decay} 
 n^{-\frac{c_1(p) \beta}{\log(1+d/\log(n))}}
  \leq a_n(\Id: G^{p,\beta,p}(\T^d) \to L_2(\T^d)) \leq
  n^{-\frac{c_2 \beta}{\log(1+d/\log(n))}} \leq n^{-\frac{c_2(p) \beta}{\log(1+d)}}.
\end{align}
We see that the dimension $d$ affects the polynomial decay in $n$ only logarithmically. 
In information-based complexity, such a decay behavior is called \emph{quasi-polynomial}.
This observation is highly remarkable for the following reason. The preasymptotic
characteristics in \eqref{eq:gevrey_quasi_polynomial_decay} closely resemble the preasymptotics observed in \cite{KSU15}
for embeddings of Sobolev spaces with dominating mixed smoothness. Concretely, the recent paper \cite{KSU15}, involving
two of the present authors, studies approximation numbers of the embedding $\Id: H_{\text{mix}}^s(\T^d) \to L_2(\T^d)$,
where the Sobolev space with dominating mixed smoothness $H_{\text{mix}}^s(\T^d)$ is equipped with one of the---in the
classical sense equivalent---norms
\[ 
 \|f | H^{s,p}_{\text{mix}}(\T^d)\| := \left( \sum_{k \in \Z^d} |c_k(f)|^2 \prod_{j=1}^d (1+ |k_j|^p)^{2s/p}
\right)^{1/2}, \quad p \in \{1,2\}.
\] 
Now, in the preasymptotic range $1 \leq n \leq 4^d$, the authors of \cite{KSU15} observe
\begin{align}\label{eq:mixed_quasi_polynomial_decay} 
 2^{-s} \left(\frac{1}{2n}\right)^{\frac{s}{c_1(p) + \log(1/2+d/\log(n))}} \leq a_n ( \Id: H^{s,p}_{\text{mix}}(\T^d)
\to L_2(\T^d) ) \leq \left( \frac{e^2}{n} \right)^{\frac{c_2(p)s}{2 + \log_2 d}},
\end{align}
where $c_1(1) = 0$, $c_1(2) = 2$, and $c_2(p) = 1/p$ for $p \in \{1,2\}$, see \cite[Thm. 4.9, 4.10, 4.17]{KSU15}.

The close resemblance of \eqref{eq:gevrey_quasi_polynomial_decay} and \eqref{eq:mixed_quasi_polynomial_decay} is rather
counterintuitive. After all, the space $G^{p,s,p}(\T^d)$ contains much smoother functions than the Sobolev space with
dominating mixed regularity $H^{s,p}_{\text{mix}}(\T^d)$, which is clearly visible in the asymptotic decay, see Remark
\ref{rem:gevrey_mixed}. But apparently, the stronger notion of smoothness does not pay off in the preasymptotic range.
Let us try to gain a deeper understanding of this unexpected relationship between spaces of Gevrey type and Sobolev
spaces of dominating mixed smoothness. From the simple estimate $\prod_{j=1}^d (1+|k_j|^p)^{s/p} \leq \exp(s/p
\|k\|_p^p)$ we conclude that we have the norm-one embedding
\begin{align}\label{eq:gevrey_mixed_norm_one_embedding}
 G^{p,s/p,p}(\T^d) \hookrightarrow H^{s,p}_{\text{mix}}(\T^d).
\end{align}
Hence, the lower bound in \eqref{eq:gevrey_quasi_polynomial_decay}, with $\beta=s/p$, yields a lower bound for the
approximation numbers $a_n ( \Id: H^{s,p}_{\text{mix}}(\T^d) \to L_2(\T^d) )$, which is only slightly worse than
\eqref{eq:mixed_quasi_polynomial_decay} with regard to the polynomial decay in $n$. Note that
\eqref{eq:mixed_quasi_polynomial_decay} has been obtained by doing the combinatorics explicitly for this special
situation, whereas \eqref{eq:gevrey_quasi_polynomial_decay} followed immediately from the characterization provided by
Theorem \ref{res:approx_entropy_mono} and the known behavior of the entropy numbers $\varepsilon_n(\id: \ell_p^d \to
\ell_\infty^d)$, see Proposition \ref{res:entropy_p}.

In view of the norm-one embedding \eqref{eq:gevrey_mixed_norm_one_embedding}, the surprising part in fact is that the
upper bound in \eqref{eq:mixed_quasi_polynomial_decay} is not substantially worse than
\eqref{eq:gevrey_quasi_polynomial_decay}. For the simplest case $p=1$ and $s=1$, there is a good explanation in terms of
grid numbers. The interested reader will find it easy to generalize this to $s > 1$. Consider the grid numbers of the
$\ell_1^d$-ball
\[ 
 \ln(r) B_1^d = \{x \in \R^d: \exp(\|x\|_1) \leq r \}
\]
and the hyperbolic cross
\[ 
 \mathcal H_r^d := \{ x \in \R^d: \prod_{j=1}^d (1+|x_j|) \leq r \}.
\]
The first determine the behavior of the approximation numbers $a_n(\Id: G^{1,1,1}(\T^d) \to L_2(\T^d))$, since $G(\ln(r)
B_1^d) = \sharp\{ k \in \Z^d: w_{1,1,1}(k) \leq r \}$ (recall the considerations made in Section
\ref{sec:approximation_numbers}). The latter determine the approximation numbers $a_n ( \Id: H^{s,p}_{\text{mix}}(\T^d)
\to L_2(\T^d) )$. We will show now that these grid numbers behave sufficiently similar for $1\leq r \leq 2^d$. Essential
ingredient of the proof is the observation that, for $l \leq \log_2(r)$, the projections $P_l\ln(r) B_1^d$ and
$P_l\mathcal H_r^d$  have similar volumes, where $P_l: \R^d \to \R^l$, $(x_1,\dots,x_d) \mapsto (x_1,\dots,x_l)$.

\begin{lemma}\label{res:gevrey_mixed_grid_numbers}
Let $1 \leq r \leq 2^d$. Then, we have
\[ 
 G(\ln(r) B_1^d) \leq G(\mathcal H_r^d) \leq r G(c\ln(r) B_1^d),
\]
where $c=1+1/\ln(2)$.
\end{lemma}

\begin{proof}
The left-hand side follows trivially from $\prod_{j=1}^d (1+|x_j|) \leq \exp(\|x\|_1)$. For the right-hand side
estimate, we first note that
\[
 G(\mathcal H_r^d) = 1 + \sum_{l=1}^{\log_2(r)} 2^l {d \choose l} A(r,l),
\]
where $A(r,l) := \sharp \{k \in \N^l: \prod_{j=1}^d (1+k_j) \leq r\}$, see \cite[Lem. 3.1]{KSU15}. Further, for $2^l
\leq r$, we have $A(r,l) \leq v_l(r)$, where $v_l(r) = \vol(\mathcal H_r^l \cap \{ x \in \R^l: x_j \geq 1\})$, and
$v_l(r) \leq r \frac{(\ln(r))^{l-1}}{(l-1)!}$, see \cite[Lem. 3.2]{KSU15}. Consider now
\begin{align*}
 G(r,l) &:= \{ k \in \N^l: \exp(\|k\|_1) \leq r\},\\
 w_l(r) &:= \vol(\ln(r)B_1^l \cap \{x\in \R^l: x_j \geq 0\}) = \frac{(\ln(r))^l}{l!}.
\end{align*}
For $k \in \Z^l$ let $Q^l_k = k + [0,1]^l$. From the two-sided set inclusion
\[ 
 \{x \in \R^l: x_j \geq 1, \; \exp(\|x\|_1) \leq r \} \subset \bigcup_{k \in G(r,l)} Q_k^l \subset \{ x \in \R^l: x_j
\geq 0, \; \exp(\|x\|_1) \leq r  \}
\]
we obtain, by taking volumes and a change of variables, the two-sided estimate
\[ 
 w_l(r/e^l) \leq \sharp G(r,l) \leq w_l(r).
\]
Now, using $l \leq \ln(r)/\ln(2)$, we observe
\[ 
 A(r,l) \leq v_l(r) = r w_{l-1}(r) \leq r w_l(r) \leq r \sharp G(e^l r, l) \leq r \sharp G(r^c,l).
\]
It remains to note that $G(c\ln(r) B_1^d) = 1+\sum_{l=1}^{\log_2(r)} 2^l {d \choose l} \sharp G(r^c,l)$, which can be
seen easily by adopting the proof of \cite[Lem. 3.1]{KSU15}.
\end{proof}

\begin{remark}
The direct preasymptotic calculations made in \cite[Thm. 4.9]{KSU15} for the Sobolev space of dominating mixed
smoothness can be adopted for the space of Gevrey type $G^{1,s,1}(\T^d)$ using the elements introduced in the proof of
Lemma \ref{res:gevrey_mixed_grid_numbers}. This yields, for $1 \leq n \leq 2^d$, the estimate from above
\[ 
 a_n(\Id: G^{1,s,1}(\T^d) \to L_2(\T^d)) \leq \Big( \frac{e^2}{n}\Big)^{\frac{s}{1+\log_2(d)}}.
\]
Compare with \eqref{eq:gevrey_quasi_polynomial_decay}.
\end{remark}

\begin{remark}\label{rem:gevrey_mixed}
In contrast to the preasymptotic range, the approximation numbers $a_n(\Id: G^{p,s/p,p}(\T^d) \to L_2(\T^d))$ and
$a_n(\Id: H_{\text{mix}}^s(\T^d) \to L_2(\T^d))$ behave asymptotically completely different. On one side, we have the
well-known result
\[
a_n(\Id: H_{\text{mix}}^s(\T^d) \to L_2(\T^d)) \asymp_d n^{-s} (\ln n)^{(d-1)s}
\] 
for the Sobolev space with dominating mixed smoothness, see \cite{KSU15} and the references therein. On the other side,
we learn from Theorem \ref{res:gevrey} (iii) that
\[ 
 q_1^{-s d n^{p/d}} \leq a_n(\Id: G^{p,\beta,p}(\T^d) \to L_2(\T^d)) \leq q_2^{-s d n^{p/d}},
\]
where $q_1=\exp(c_1/p), q_2 = \exp(c_2/p)$.
\end{remark}

\subsection{Preasymptotics for embeddings into $\mathbf{H^s}$}\label{sec:gevrey:energy}

\begin{figure}[t]
\centering
\begin{tikzpicture}
  \tikzset{node distance=3cm, auto}
  \node (H) {$G^{\alpha,\beta,p}(\T^d)$};
  \node (L) [right of =H] {$H^{s,p}(\T^d)$};
  \node (ell) [below of=H] {$\ell_2(\Z^d)$};
  \node (ell2) [right of=ell] {$\ell_2(\Z^d)$};
  \draw[->] (H) to node {$Id$} (L);
  \draw[->] (H) to node[left] {$A^{\mathbf{w}^G_{\alpha,\beta,p}}$} (ell);
  \draw[->] (ell) to node {$D^{\widetilde{\mathbf{w}}}$} (ell2);
  \draw[->] (ell2) to node[right] {$B^{\mathbf{w}_{s,p}}$} (L);
\end{tikzpicture}
\caption{Commutative diagram for the embedding $\Id: G^{\alpha,\beta,p}(\T^d) \to H^{s,p}(\T^d)$.}
\label{fig:commutative_diagram_gevrey}
\end{figure}

In this section, we consider approximation numbers of the embedding
\[ 
 \Id: G^{\alpha,\beta,p}(\T^d) \to H^{s,p}(\T^d),
\]
assuming $s \leq \beta \alpha$. Here, as before, $G^{\alpha,\beta,p}(\T^d)$ is the periodic space of Gevrey type defined
by the weight sequence $\mathbf{w}_{\alpha,\beta,p}^G$, see \eqref{eq:gevrey_weights}, and $H^{s,p}(\T^d)$ is the
isotropic periodic Sobolev space defined by the weight sequence $\mathbf{w}_{s,p}$, see \eqref{def:iso_weights}. 

For $\mathbf{w}$ an arbitrary weight sequence, recall the operators $A^{\mathbf{w}}$, $D^{\mathbf{w}}$, and $F$ defined
in \eqref{eq:seq_op},  \eqref{eq:diag_op} and \eqref{eq:fourier_op} in Section \ref{sec:approximation_numbers}. Further,
let
\begin{align*}
 B^{\mathbf{w}}: \ell_2(\Z^d) \to H^{\mathbf{w}}(\T^d), \; (\xi_k)_{n \in \Z^d} \mapsto (2\pi)^{-d/2} \sum_{k \in \Z^d}
\xi_k/w(k) e^{i k x}.
\end{align*}
We can write the embedding $ \Id: G^{\alpha,\beta,p}(\T^d) \to H^{s,p}(\T^d)$ as $\Id = B^{\mathbf{w}_{s,p}} \circ
D^{\widetilde{\mathbf{w}}} \circ A^{\mathbf{w}_{\alpha,\beta,p}^G}$, where 
$\widetilde{\mathbf{w}} = \mathbf{w}^G_{\alpha,\beta,p}/\mathbf{w}_{s,p}$,
see Figure \ref{fig:commutative_diagram_gevrey} for an illustration. At the same time, we also have $\Id = F \circ
D^{\widetilde{\mathbf{w}}} \circ A^{\widetilde{\mathbf{w}}}$. Hence, recalling the considerations made at the beginning
of Section \ref{sec:approximation_numbers}, it is clear that
\[ 
 a_n(\Id: G^{\alpha,\beta,p}(\T^d) \to H^{s,p}(\T^d)) = a_n( \Id: H^{\widetilde{\mathbf{w}}}(\T^d) \to L_2(\T^d)).
\]
Note that $\widetilde{w}(k) = \widetilde{\varphi}(\|k\|_p)$, where
\begin{align}\label{eq:phi_tilde}
 \widetilde{\varphi}(t) = \exp(\beta t^\alpha)/t^s.
\end{align}
Since we have assumed $s \leq \beta \alpha$, the function $\widetilde{\varphi}$ is monotonically increasing for all
$t\geq 1$. Consequently, we may apply Theorem \ref{res:approx_entropy_mono} and obtain the following worst-case error
estimates.

\begin{theorem}\label{res:gevrey_energy}
Let $\alpha,\beta,s > 0$, such that $s \leq \beta\alpha$, and $0<p\leq \infty$. Consider the approximation numbers
\[ 
 a_n := a_n(\Id: G^{\alpha,\beta,p}(\T^d) \to H^{s,p}(\T^d)).
\]
\begin{enumerate}[label=(\roman*)]
 \item For $1 \leq n \leq d$, we have $a_n \asymp_{\alpha,\beta,s,p} 1$.
 \item For $d \leq n \leq 2^d$, we have
 \[ 
  -\ln(a_n) + s \ln\left(\frac{\log(n)}{\log(1+d/\log(n))}\right)  \asymp_{\alpha,s,p}- 1 + \beta
\left(\frac{\log(n)}{\log(1+d/\log(n))}\right)^{\alpha/p}.
 \]
 \item For $n \geq 2^d$, we have
 \[ 
  -\ln(a_n) +  s\ln(d) +s/d\ln(n) \asymp_{\alpha,s,p} -1 + \beta d^{\alpha/p} n^{\alpha/d}.
 \] 
\end{enumerate} 
\end{theorem}

\noindent In case that $\alpha = p$, $\beta = r/p$, which is particularly interesting in view of a comparison with spaces of dominating mixed smoothness and the discussion in Subsection \ref{sec:intro:enegry}, we can rewrite Theorem \ref{res:gevrey_energy} (ii) as follows.

\begin{corollary}\label{cor:gevrey_energy}
Consider the approximation numbers
\[ 
 a_n := a_n(\Id: G^{p,r/p,p}(\T^d) \to H^{s,p}(\T^d)).
\]
For $1 \leq n \leq 2^d$, we have
\begin{equation}\begin{split} 
  &c_1(p,s) \left(\frac{\log(n)}{\log(1+d/\log(n))}\right)^{s/p} n^{-\frac{c_1(p) r}{p\log(1+d/\log(n))}}\\
  &~~~~~~~~~\leq a_n \leq
  c_2(p,s)\left(\frac{\log(n)}{\log(1+d/\log(n))}\right)^{s/p} n^{-\frac{c_2(p) r}{p\log(1+d/\log(n))}}.
  \end{split}
\end{equation}
\end{corollary}
\begin{proof}
Let $\alpha = p$, $\beta=r/p$. Then the asserted follows by Theorem \ref{res:gevrey_energy} (ii).
\end{proof}

\noindent As a last point in this section we provide the following limit result.
\begin{theorem}\label{limit:gevrey_energy}
Let $0 < p \leq \infty$, $0< \alpha < \min\{1,p\}$, and $\beta > 0$. For
\[ 
 a_n := a_n(\Id: G^{\alpha,\beta,p}(\T^d) \to H^{s,p}(\T^d)),
\]
we have
\[ 
 \lim_{n \to \infty} a_n \cdot \frac{\exp(\beta \vol(B_p^d)^{-\alpha/d} n^{\alpha/d}) \vol(B_p^d)^{s/d}}{n^{s/d}} = 1.
\]
\end{theorem}

\begin{proof}
Let $\varepsilon_n := \varepsilon_n(\id: \ell_p^d \to \ell_\infty^d)$. The general estimate \eqref{eq:general_limit_two_sided_estimate} now takes the form
\[ 
 \frac{1}{\widetilde{\varphi}(h_2(1/(2\varepsilon_n)))} \leq a_n \leq \frac{1}{\widetilde{\varphi}(h_1(1/(2\varepsilon_n)))},
\]
where $\widetilde{\varphi}$ is defined in \eqref{eq:phi_tilde} and $h_1$, $h_2$ are defined in the proof of Theorem \ref{res:general_limit}. This can be further estimated to
\[ 
 \frac{1}{(2\varepsilon_n)^s\exp(\beta h_2(1/(2\varepsilon_n))^\alpha)} \leq a_n \leq \frac{1}{(2\varepsilon_n)^s\exp(\beta h_1(1/(2\varepsilon_n))^\alpha)}.
\]
Writing $x_n = n^{1/d} \vol(B_p^d)^{-1/d}$, it is easy to see that plugging in the estimates of Lemma \ref{res:entropy_numbers_large_n} leads to
\[ 
 \left(1- \frac{2^{1-p}d^{1/p}}{x_n}\right)^{s/p} \exp(\beta(x_n^\alpha - h_2(x_n)^\alpha)\leq a_n \frac{\exp(\beta \vol(B_p^d)^{-\alpha/d} n^{\alpha/d}) \vol(B_p^d)^{s/d}}{n^{s/d}}
\]
and
\[ 
 a_n \frac{\exp(\beta \vol(B_p^d)^{-\alpha/d} n^{\alpha/d}) \vol(B_p^d)^{s/d}}{n^{s/d}} \leq \exp(\beta(x_n^\alpha - h_3(x_n)^\alpha)),
\]
where $h_3$ is defined in the proof of Theorem \ref{limit:gevrey}. It remains to apply the arguments which we already used in the proof of Theorem \ref{limit:gevrey}.
\end{proof}

\section{Tractability analysis}\label{sec:tractability}

We conclude this paper with a tractability discussion. The tractability results follow more or less immediately from the
worst-case error bounds which we have derived in the preceding sections.

\begin{theorem}\label{res:tractability_iso}
Let $s > 0$ and $0<p\leq \infty$. Then the approximation problem
\[ 
 \Id: H^{s,p}(\T^d) \to L_2(\T^d)
\]
\begin{enumerate}[label=(\roman*)]
 \item suffers from the curse of dimensionality iff $p = \infty$ (for all $s > 0$),
 \item does not suffer from the curse of dimensionality iff $p < \infty$ and $s > 0$,
 \item is intractable iff $p < \infty$ and $s \leq p$,
 \item is weakly tractable iff $p < \infty$ and $s > p$.
\end{enumerate}
\end{theorem}

\begin{theorem}\label{res:tractabibility_gevrey}
Let $\alpha, \beta > 0$ and $0<p\leq \infty$. Then the approximation problem
\[ 
 \Id: G^{\alpha,\beta,p}(\T^d) \to L_2(\T^d)
\]
is quasi-polynomial tractable if and only if $\alpha \geq p$.
\end{theorem}

\noindent Before we turn to the proofs of Theorems \ref{res:tractability_iso}, \ref{res:tractabibility_gevrey}, let us
stress at this point that for the situations discussed here, the decay of approximation numbers in the preasymptotic
range determines the tractability. This is a particularly interesting observation regarding the isotropic Sobolev space.
As we have already pointed out in Section \ref{sec:intro}, the asymptotic decay
\[ 
 a_n(\Id: H^{s,p}(\T^d) \to L_2(\T^d)) \asymp_{s,p,d} n^{-s/d}
\]  
is often considered a typical indicator for the curse of dimensionality. However, as Theorem \ref{res:tractability_iso}
shows, the approximation problem suffers only from the curse of dimensionality in the strict sense of information-based
complexity when we equip the isotropic Sobolev space with the norm $\|\cdot \mid H^{s,\infty}(\T^d)\|$. Otherwise, the
approximation problem is weakly tractable, despite the bad asymptotic decay $n^{-s/d}$. For $p=1,2,2s$ this has already
been observed in \cite{KSU13}, see Remark \ref{rem:KSU13}. Concerning spaces of Gevrey type, it is no surprise in light
of Section \ref{sec:preasymptotic_similarity} that we obtain a similar tractability as has been observed for Sobolev
spaces with dominating mixed regularity in \cite{KSU15}. For some further remarks on the tractability of Gevrey
embeddings, see Remark \ref{rem:Oesis}.

For the proof of Theorem \ref{res:tractability_iso} we have to translate the bounds of Theorem \ref{res:approx_numbers}
into bounds for the information complexity. These bounds are given in Lemma \ref{res:icompl_iso}. We omit the proof,
which is technical and lengthy but requires only standard arguments.

\begin{lemma}\label{res:icompl_iso}
For $s > 0 $ and $0 < p < \infty$, consider the information complexity
\[
 n(\varepsilon,d) = \min\{n \in \N: a_n(\Id: H^{s,p}(\T^d) \to L_2(\T^d)) \leq \varepsilon\}.
\]
\begin{enumerate}[label=(\roman*)]
 \item From above, we have the bounds
 \begin{align*}
 \log n(\varepsilon,d) \lesssim_{s,p}
 \begin{cases}
 \log (d) &: \varepsilon^U_1 \leq \varepsilon \leq 1\\
 \log (d) \; (1/\varepsilon)^{p/s} &: \varepsilon_2^U \leq \varepsilon \leq \varepsilon_1^U\\
 \log(1/\varepsilon) \; (1/\varepsilon)^{p/s} &: \varepsilon_3^U(\gamma) \leq \varepsilon \leq \varepsilon_2^U\\
 \log(1/\varepsilon) \; (1/\varepsilon)^{\frac{p\gamma}{s(p+\gamma)}} &: \varepsilon \leq \varepsilon_3^U(\gamma)
 \end{cases}
 \end{align*}
 where $\gamma \geq 0$ and
 \begin{align*}
  \varepsilon_1^U := C_{s,p} \left[ \frac{\log(1+d/\log d)}{\log d}\right]^{s/p}, \quad
  \varepsilon_2^U := C_{s,p} d^{-s/p}, \quad
  \varepsilon_3^U := C_{s,p} 2^{-s} d^{-s(1/p+1/\gamma)}, 
 \end{align*}
 The constant $C_{s,p}$ is the same as in the upper bound of Theorem \ref{res:approx_numbers}.
 \item From below, we have the bound
 \[ 
  \log n(\varepsilon,d) \gtrsim_{s,p} (1/\varepsilon)^{p/s} \quad \text{ for } \; \varepsilon_2^L \leq \varepsilon \leq
\varepsilon_1^L,
 \]
 where
 \begin{align*}
  \varepsilon_1^L := c_{s,p} \left[ \frac{\log(1+d/\log(d))}{\log(d)}\right]^{s/p}, \quad
  \varepsilon_2^L := c_{s,p}  2^{-s} (1/d)^{s/p}.
 \end{align*}
 The constant $c_{s,p}$ is identical to the one in the lower bound of Theorem \ref{res:approx_numbers}.
\end{enumerate}
\end{lemma}

\begin{proof}[Proof of Theorem \ref{res:tractability_iso}]
~ 
\begin{enumerate}[label=(\roman*)]
 \item For $n \leq 2^d$, Theorem \ref{res:approx_numbers_p_infty} states that   $a_n = 1$. Hence, we have
$n(\varepsilon,d) \geq 2^d$ for all $\varepsilon < 1$ and the problem suffers from the curse of dimensionality.
 \item We show that for there is an $\varepsilon > 0$ such that $n(\varepsilon,d)$ is polynomial in $d$. Fix some
$\varepsilon > \varepsilon_2^U$. By Lemma \ref{res:icompl_iso} (i), there is $\tilde C_{s,p} > 0$ such that
 \[ 
  n(\varepsilon,d) \leq d^{\tilde C_{s,p} (1/ \varepsilon)^{p/s}}.
 \]
 Since $\varepsilon > \varepsilon_2^U$ the above estimate holds true for all $d > C_{s,p}^{p/s} (1/\varepsilon)^{p/s}$.
Consequently, the problem cannot suffer from the curse of dimensionality.
 \item To prove intractability is suffices that there is a sequence $(\varepsilon_i, d_i)_{i \in \N}$ such that the
limit in \eqref{eq:wtrac} does not exist. Let $(d_i)_{i \in \N}$ with $d_i \in \N$ and $d_i \to \infty$ for $i \to
\infty$. Let $\varepsilon_2^L \leq \varepsilon_i \leq c_{s,p} (1/d_i)^{s/p}$. Then,
 \[ 
  c_{s,p}^{p/s} 2^{-p} (1/\varepsilon_i)^{p/s} \leq d_i \leq c_{s,p}^{p/s} (1/\varepsilon_i)^{p/s},
 \]
 and thus
 \[ 
  \frac{\log n(\varepsilon_i,_id)}{d_i+ 1/\varepsilon_i} \geq c_{s,p} \frac{(1/\varepsilon_i)^{p/s}}{d+1/\varepsilon_i}
\geq c_{s,p} \frac{(1/\varepsilon_i)^{p/s}}{c_{s,p}^{p/s} (1/\varepsilon_i)^{p/s}  + 1/\varepsilon_i}. 
 \]
 Finally, since $p/s \geq 1$, we have $1/\varepsilon_i \leq (1/\varepsilon_i)^{p/s}$ and thus
 \[ 
 \frac{\log n(\varepsilon_i,d_i)}{d_i + 1/\varepsilon_i} \geq \frac{c_{s,p}}{c_{s,p}^{p/s} + 1} > 0 \quad \text{for all
} i \in \N.
 \]
 In consequence, the problem is not weakly tractable and must be intractable.
 \item We have to show that the information complexity grows slower than both $2^{1/\varepsilon}$ and $2^d$. Put
$x=1/\varepsilon + d$. Since both $1/\varepsilon \leq x$ and $d \leq x$, we have for all $\varepsilon$ and all $d$ that
 \[ 
  \log n(\varepsilon,d) \leq \tilde C_{s,p} \log(x) x^{p/s}.
 \]
 Hence, $\lim_{x \to \infty} \log n(\varepsilon,d) / x = 0$ as $p < s$.
\end{enumerate}\end{proof}

The tractability analysis for the approximation problem $\Id: G^{\alpha,\beta,p}(\T^d) \to L_2(\T^d)$ can be reduced to
the tractability analysis for the problem $\Id: H^{\mathbf{w}_p}(\T^d) \to L_2(\T^d)$. Basis is the
following general observation.

\begin{lemma}\label{res:icompl_charac}
Let $\mathbf{w}$ be an arbitrary weight sequence and let
\[ 
 n^{\mathbf{w}}(\varepsilon,d) := \min \{n \in \N: a_n(\Id: H^{\mathbf{w}}(\T^d) \to L_2(\T^d)) \leq \varepsilon \}.
\]
For $\varphi: [0,\infty) \to [0,\infty)$ monotonically increasing, consider the weight sequence $\varphi(\mathbf{w})$
given by $\varphi(\mathbf{w})(0) := 1$ and $\varphi(\mathbf{w})(k) := \varphi(w(k))$ for $k \in \Z^d \setminus \{0\}$.
We have
\[ 
 n^{\varphi(\mathbf{w})}(\varepsilon,d) = n^{\mathbf{w}}(1/\varphi(1/\varepsilon),d).
\]
\end{lemma}

\begin{proof}
Let $(\sigma_n)_{n \in \N}$ and $(\gamma_n)_{n \in \N}$ be the non-increasing rearrangements of\\ $1/\mathbf{w}$ and 
$1/\varphi(\mathbf{w})$, respectively. Then, using \eqref{eq:rearr_mon}, we obtain
\begin{align*} 
 n^{\varphi(\mathbf{w})}(\varepsilon,d) &= \min \{n \in \N: \gamma_n \leq \varepsilon\}
                  = \min \{n \in \N: 1/\varphi(1/\sigma_n) \leq \varepsilon \}\\
                  &= n^{\mathbf{w}}(1/\varphi^{-1}(1/\varepsilon),d).
\end{align*}
\end{proof}

\begin{proof}[Proof of Theorem \ref{res:tractabibility_gevrey}]
With $\varphi(t) = \exp(\beta t^\alpha)$ and $\gamma = p$, Lemma \ref{res:icompl_charac} in combination with Lemma
\ref{res:icompl_iso} (i) yields
\begin{align*}
\ln n(\varepsilon,d) \lesssim_{\alpha,\beta,p}
\begin{cases}
 \ln(d) &: \tilde \varepsilon_1^U \leq \varepsilon \leq 1\\
  \ln(d) \ln(1/\varepsilon)^{p/\alpha} &: \tilde\varepsilon_2^U \leq \varepsilon \leq \tilde\varepsilon_1^U\\
  \ln(\ln(1/\varepsilon)) \ln(1/\varepsilon)^{p/\alpha} &:\tilde\varepsilon_3^U \leq \varepsilon \leq
\tilde\varepsilon_2^U\\
  \ln(\ln(1/\varepsilon)) \ln(1/\varepsilon)^{p/(2\alpha)} &: \varepsilon \leq \tilde\varepsilon_3^U 
\end{cases}
\end{align*}
where $\tilde\varepsilon_i^U = 1/\varphi(1/\varepsilon_i^U)$. Since in the third case we may estimate
$\ln(\ln(1/\varepsilon)) \lesssim_{\alpha,\beta,p} \ln(d)$ due to $\tilde \varepsilon_3^U \leq \varepsilon$ and in the
forth case we may estimate $\ln(\ln(1/\varepsilon)) \lesssim_{\alpha,\beta,p} \ln(1/\varepsilon)^{p/(2\alpha)}$, we
obtain
\[ 
 \ln n(\varepsilon,d) \lesssim_{\alpha,\beta,p} \ln(d) \ln(1/\varepsilon)^{p/\alpha}
\]
for all $0< \varepsilon \leq 1$ and $d \in \N$, which leads to quasi-polynomial tractability if $\alpha \geq p$. That
$\alpha  \geq p$ is also a necessary condition for quasi-polynomial tractability follows immediately by Lemma
\ref{res:icompl_charac} and Lemma \ref{res:icompl_iso} (ii).
\end{proof}

\begin{remark}\label{rem:KSU13}
The tractability of approximating the identity $\Id: H^{s,p}(\T^d) \to L_2(\T^d)$ by finite-rank operators has already
been studied in \cite{KSU13} for $p=1$, $p=2$, and $p=2s$. In the case $p=2$, however, the authors could not show
whether the problem is intractable or weakly tractable when $1 < s \leq 2$, see \cite[Thm. 5.5, Cor. 5.7]{KSU13}. The
results from Section \ref{sec:approximation_numbers} allow to close this gap and furthermore to reproduce all
tractability results obtained in \cite{KSU13}. For a different proof that allows to close the gap, we refer to the
recent paper \cite{SW14}.
\end{remark}

\begin{remark}
Concerning the standard notions of tractability, asking for compressibility of frequency vectors ($0 < p \leq 1$) only
has the effect that we need less smoothness to obtain weak tractability, see Theorem \ref{res:tractability_iso}, (iv).
To get a comprehensive understanding of the effect of compressibility, we need two additional notions of tractability
introduced only recently. For $\alpha, \beta > 0$ a problem is called \emph{($\alpha$, $\beta$)-weakly tractable}
\cite{SW14} if
\[ 
 \lim\limits_{1/\varepsilon+d\to\infty} \frac{\log n(\varepsilon,d)}{1/\varepsilon^\alpha + d^\beta} = 0\,.
\]
A problem is called \emph{uniformly weakly tractable} \cite{S12} if it is ($\alpha$,$\beta$)-weakly tractable for all
$\alpha$,$\beta > 0$.
From Lemma \ref{res:icompl_iso} we can conclude that the approximation problem $\Id: H^{s,p}(\T^d) \to L_2(\T^d)$ is
\emph{($\alpha$, $\beta$)-weakly tractable} for $\alpha > p/s$ and all $\beta > 0$ (which has also been observed in
\cite[Thm 4.1]{SW14}). Hence, if we impose a very strong compressibility constraint---which means that $p$ gets
small---then we have almost \emph{uniform weak tractability}.
\end{remark}


\begin{remark}\label{rem:Oesis}
The recent paper \cite{DKPW13} studies the tractability of approximating embeddings $\Id: H^{\mathbf{w}}(\T^d) \to
L_2(\T^d)$ by operators of finite-rank for weight sequences $\mathbf{w}$ of the form $w(k) =  \omega^{\sum_{j=1}^d a_j
|k_j|^{b_j}}$, $k \in \Z^d$, where $\omega > 1$, $0 < a_1 \leq a_2 \leq a_3 \leq \dots$ and $\inf b_j > 0$.
Let
\[ 
 n^{\mathbf{w}}(\varepsilon,d) := \min \{n \in \N: a_n(\Id: H^{\mathbf{w}}(\T^d) \to L_2(\T^d)) \leq \varepsilon \}
\]
be the information complexity of the approximation problem. In \cite{DKPW13} it is studied under which conditions a
modified, stronger notion of weak tractability is satisfied, namely
\begin{align}\label{eq:modifiet_wt} 
 \lim_{\ln(1/\varepsilon)+d \to \infty} \frac{\ln n^{\mathbf{w}}(\varepsilon,d)}{\ln(1/\varepsilon)+d} = 0.
\end{align}

The Gevrey weights $\mathbf{w}^G_{\alpha,\beta,p}$, defined in \eqref{eq:gevrey_weights}, fit into the setting of \cite{DKPW13} if $\alpha = p$ (by
choosing $a_1=a_2=\cdots=\beta$ and $b_1=b_2=\cdots=p$). From \cite[Thm. 1]{DKPW13} it is immediately clear that the approximation problem $\Id: G^{\alpha,\beta,p}(\T^d) \to
L_2(\T^d)$ is not weakly tractable in the above sense if $\alpha=p$. What more can be said? From Lemma
\ref{res:icompl_charac} we get
\[ 
 \lim_{1/\varepsilon + d} \frac{\ln n^{\mathbf{w}_{\alpha,p}}(\varepsilon,d)}{d + 1/\varepsilon} = \lim_{1/\varepsilon +
d} \frac{\ln n^{\mathbf{w}^G_{\alpha,\beta,p}}(\varepsilon,d)}{d + \ln(1/\varepsilon)}.
\]
Hence the approximation problem $\Id: G^{\alpha,\beta,p}(\T^d) \to L_2(\T^d)$ is weakly tractable in the modified sense
\eqref{eq:modifiet_wt} if and only if the approximation problem $\Id: H^{\alpha,p}(\T^d) \to L_2(\T^d)$ is weakly
tractable in the classical sense, which is the case if and only if $\alpha > p$. We conclude that weak tractability in the modified sense \eqref{eq:modifiet_wt} is almost
equivalent to quasi-polynomial tractability for the approximation problem
$\Id: G^{\alpha,\beta,p}(\T^d) \to L_2(\T^d)$, cf. Theorem \ref{res:tractabibility_gevrey}.

As a final remark let us point out that other than claimed in \cite{DKPW13} the space $H^{\mathbf{w}}(\T^d)$ consists of
analytic functions if and only if $\inf b_j \geq 1$. The proof provided in \cite[Section 10]{DKPW13} is wrong, and even under the additional assumption $\inf b_j \geq 1$ incomplete as it only shows convergence of the Taylor expansion. For a correct proof, see \cite{IKLP15}.
\end{remark}

\section*{Acknowledgments}
The authors would like to thank David Krieg, Erich Novak, Winfried Sickel, Markus Weimar, and Henryk Wo{\'z}niakowski for fruitful discussions. Furthermore, they would like to thank two anonymous referees for their valuable comments
on an earlier version of this manuscript. Tino Ullrich and Sebastian Mayer gratefully acknowledge support by the German Research
Foundation (DFG) Ul-403/2-1 as well as the Emmy-Noether programme, Ul-403/1-1. Thomas K\"uhn
was supported in part by the Spanish Ministerio de Econom\'ia y Competitividad (MTM2013-42220-P).

\bibliographystyle{abbrv}
\bibliography{references}
\end{document}